\documentclass[11pt,a4paper]{article}
\usepackage[utf8]{inputenc}
\usepackage{amsmath, amssymb, amsthm}
\usepackage{graphicx,hyperref,geometry}
\usepackage{a4wide}
\usepackage{xcolor,verbatim}
\usepackage{tikz}
\usepackage{tkz-graph}
\usepackage{float}
\usetikzlibrary{positioning}
\newtheorem{theorem}[subsection]{Theorem}

\newtheorem{corollary}[subsection]{Corollary}

\newtheorem{proposition}[subsection]{Proposition}

\newtheorem{result}[subsection]{Result}

\newtheorem{problem}[subsection]{Problem}
\newtheorem{observation}[subsection]{Observation}

\newcommand{\dom}{\mathrm{dom}}
\newcommand{\DOM}{\mathrm{DOM}}
\newcommand{\cC}{\mathcal{C}}
\newcommand{\cF}{\mathcal{F}}

\title{Characterization of graphs with orientable total domination number equal to $|V|-1$}
\author{Zoltán L. Blázsik\thanks{Bolyai Institute, University of Szeged, Aradi v\'ertan\'uk tere 1, 6720
Szeged, Hungary. \\ University of Johannesburg Auckland Park, 2006 South Africa, E-mail: \url{blazsik@server.math.u-szeged.hu},}
\and
Leila Vivien Nagy\thanks{Eötvös Loránd University, Pázmány Péter sétány 1/A, Budapest, Hungary.
}
}

\begin{document}
\maketitle
\begin{abstract}
In a directed graph $D$, a vertex subset $S\subseteq V$ is a total dominating set if every vertex of $D$ has an in-neighbor from $S$. A total dominating set exists if and only if every vertex has at least one in-neighbor. We call the orientation of such directed graphs valid. The total domination number of $D$, denoted by $\gamma_t(D)$, is the size of the smallest total dominating set of $D$. For an undirected graph $G$, we investigate the upper (or lower) orientable total domination number of $G$, denoted by $\DOM_t(G)$ (or $\dom_t(G)$), that is the maximum (or minimum) of the total domination numbers over all valid orientations of $G$. We characterize those graphs for which $\DOM_t(G)=|V(G)|-1$, and consequently we show that there exists a family of graphs for which $\DOM_t(G)$ and $\dom_t(G)$ can be as far as possible, namely $\DOM_t(G)=|V(G)|-1$ and $\dom_t(G)=3$. 
\end{abstract}

\section{Introduction} 
Given an undirected graph $G=(V,E)$ with vertex set $V$ and edge set $E$. The graphs in this paper are finite (i.e. $|V|$ is finite) and without loops or multiple edges. The \emph{neighborhood} $N(v)$ of a vertex $v$ is the set of vertices adjacent to $v$, i.e. $N(v)=\{u\ |\ uv\in E\}$. The vertices in $N(v)$ are the \emph{neighbors} of $v$. The \emph{degree} $d(v)$ of $v$ is defined as the number of neighbors of $v$. The vertex $v$ is \emph{isolated} if $d(v)=0$, and \emph{full} if every other vertex is a neighbor of $v$, i.e. $d(v)=|V|-1$. A graph is \emph{isolate-free} if there are no isolated vertices in it. The \emph{minimum degree} and the \emph{maximum degree} of $G$ is denoted by $\delta(G)$ and $\Delta(G)$, respectively. A vertex $u\in{V}$ \emph{dominates} $v\in{V}$ if $uv\in{E}$, or in other words if $v\in N(u)$. Then $v$ is \emph{dominated by} $u$. A vertex subset $S\subseteq V$ \emph{dominates} $T\subseteq V$ if and only if $\cup_{v\in S}~ N(v)=T$. A vertex subset $S\subseteq V$ is a \emph{dominating set} of $G$ if each vertex in $V\setminus S$ is dominated by a vertex in $S$, or equivalently if $\cup_{v\in S}~N(v)=V$. A dominating set that does not contain any other dominating set as a proper subset is a \emph{minimal dominating set}.  
The \emph{domination number} of $G$, denoted by $\gamma(G)$, is the minimum cardinality of a dominating set of $G$. 

A set $S\subseteq V$ is a \emph{total dominating set} of $G$ if every vertex in $V$ is dominated by a vertex from $S$, i.e. $\displaystyle\cup_{v\in S}~N(v) = V$. Note that a graph $G$ has a total dominating set if and only if it is isolate-free. A total dominating set is a \emph{minimal total dominating set} if it does not contain any total dominating set as a proper subset. The \emph{total domination number} of $G$, denoted by $\gamma_t(G)$, is the minimum cardinality of a total dominating set of $G$. Since total domination is in the main focus of this paper, let us suppose that every graph is isolate-free in the sequel.

The concept of domination can be applied for directed graphs as well in the following way. Given a directed graph $D=(V,A)$ with vertex set $V$, and arc set $A$. The \emph{in-neighborhood} $N^D_{-}(v)$ of a vertex $v\in V$ in $D$ is the collection of those vertices $u\in V$ such that $uv\in A$. The \emph{out-neighborhood} $N^D_{+}(v)$ of a vertex $v\in V$ in $D$ is the collection of those vertices $u\in V$ such that $vu\in A$. Similarly, the \emph{in-degree} and \emph{out-degree} of a vertex $v$, denoted by $d^D_{-}(v)$ and $d^D_{+}(v)$, is defined in the following way: $d^D_{-}(v) = |N^D_{-}(v)|$ and $d^D_{+}(v) = |N^D_{+}(v)|$. The vertex $u\in{V}$ \emph{dominates} $v\in{V}$ if $u \in {N^D_{-}(v)}$. A set $S\subseteq V$ is a \emph{dominating set} of $D$ if each vertex in $V\setminus S$ is dominated by a vertex from $S$. A set $S\subseteq V$ is a \emph{total dominating set} of $D$ if every vertex in $V$ is dominated by a vertex in $S$. The \emph{domination number} (or \emph{total domination number}) is the minimum cardinality of a dominating set (or total dominating set) of $D$, and it is denoted by $\gamma(D)$ (or $\gamma_t(D)$). Note that a directed graph has a total dominating set if and only if there are no vertices whose in-degree is equal to 0. Therefore total domination can only be interpreted in such directed graphs where the in-degree of any vertex is at least 1. For the interested reader we refer to excellent books on this topic \cite{HHH1,HHH2,HHH3,HY}.

Let $G=(V,E)$ be an undirected graph. An \emph{orientation} of $G$ is a directed graph $D(V,A)$ which has the same vertex set as $G$ and every edge from $G$ is oriented in one of the two possible ways. We call an orientation \emph{valid} if the in-degree of every vertex is at least 1, thus the notion of a total dominating set can be interpreted and always exists. What kind of graphs admit a valid orientation? It is straightforward to see that every connected component of $G$ must have at least as many edges as vertices because otherwise the non-zero in-degree condition would fail. Thus every connected component must contain a cycle. However in each  connected component, if we fix a spanning tree $T$ and add another edge $e$ from this component then it will introduce a cycle $C$ in $T\cup \{e\}$. If we orient the edges of $C$ to be a circuit and the rest of the edges of $T\cup \{e\}$ is oriented outwards from the cycle then every vertex of this component will have at least one in-neighbor. Hence those graphs that admit a valid orientation have a complete characterization, namely in each of their connected components they must contain at least one cycle. Let us denote the class of these graphs with $\cC$. In the sequel, we will always consider undirected graphs which belong to $\cC$.

There are recent papers \cite{ABKKR,CH} that are focusing on the \emph{orientable domination number} of an undirected graph $G$ from class $\cC$. The orientable domination number is denoted by $\DOM(G)=\max\{\gamma(D):~D~\mathrm{is~a~valid~orientation~of~}G\}$. The analogous concept was introduced earlier for total domination in \cite{HLY1,HLY2}. The \emph{(upper) orientable total domination number} is defined to be $\DOM_t(G)=\max\{\gamma_t(D):~D~\mathrm{is~a~valid~orientation~of~}G\}$, and the \emph{(lower) orientable total domination number} is defined to be $\dom_t(G)=\min\{\gamma_t(D):~D~\mathrm{is~a~valid~orientation~of~}G\}$. It was also shown in those early papers that $\DOM_t(G)=|V(G)|=\dom_t(G)$ if and only if $G$ is a cycle. Recently, the upper and lower orientable total domination numbers were investigated by Anderson et al. \cite{ADJK}, and they characterized those graphs of girth at least 7 for which $\DOM_t(G)=\dom_t(G)$, and also characterized those graphs for which $\dom_t(G)=|V(G)|-1$. They introduced a family of graphs, denoted by $\cF$, that is the collection of those graphs that can be obtained from a cycle $C_k$, $k\ge 3$, and a path $P_{\ell}$, $\ell\ge 2$, by identifying one of the leaves of the path with an arbitrary vertex of the cycle.

\begin{result}\label{r:char}
    Let $G\in \cC$ be a connected graph. Then $\dom_t(G)=|V(G)|-1$ if and only if $G\in \cF \cup \{K_4,K_{2,3},K_4-e\}$.
\end{result}

They wondered whether there exists a graph $G$ such that $\dom_t(G)<\DOM_t(G)=|V(G)|-1$ holds. As they have already pointed out, if $\dom_t(G)=|V(G)|-1$ then $\DOM_t(G)=|V(G)|-1$ follows as well. Hence our task is to find all such graphs for which $\dom_t(G)<\DOM_t(G)=|V(G)|-1$ holds. They pose the following question as {\bf Problem 1.} in \cite{ADJK}, and also suggested that the solution to the question might results in some graphs such that $\dom_t(G)<|V(G)|-1=\DOM_t(G)$. 

\begin{problem}\label{p:p1}
    Characterize all graphs where $\DOM_t(G)=|V(G)|-1$.
\end{problem}

Let us introduce further families of connected graphs in order to formulate our main result. A connected graph $G\in \cC$ belongs to the family $\cF_2$ if $\delta(G)\ge 2$ and $G$ contains $\ell\ge 1$ vertex disjoint cycles and one additional vertex $s$. The further edges of $G$ must be incident with $s$ such that $s$ is adjacent to at least one vertex from each cycle component while $d(s)\ge 2$.
Let $\cF_1$ denote a family that generalises the family $\cF$ from \cite{ADJK} in the following sense. If a connected graph $G\in\cC$ belongs to $\cF_1$, then $G$ must satisfy $\delta(G)=1$, and $G$ has to have only one vertex $s$ of degree $1$. Furthermore, $G$ must have a path $sw_1 w_2 \dots w_k$ of length $k\ge 1$ from $s$ to $w_k$ and $G$ might have further $\ell\ge 0$ vertex-disjoint cycle components. All of the additional edges must be incident with $w_k$ such that $w_k$ is adjacent to at least one vertex from each cycle component and $d(w_k)\ge 2$. Observe that if $\ell=1$ and $d(w_k)=2$ then these graphs are exactly the elements of $\cF$. Moreover, if we extend a graph $G\in\cF_2$ with a path by identifying one of the endpoints of the path with $s$ and may add some further edges between the inner vertices of the path and $s$ then we can construct any graph from $\cF_1$ that has at least $1$ cycle component, i.e. $\ell\ge 1$. Hence $\cF\subseteq \cF_1$ and $K_4$, $K_{2,3}$, and $K_4-e$ belongs to $\cF_2$.

Let $\cF_3$ denote another family of connected graphs that consists of graphs $G\in\cC$ that can be constructed from a graph $R\in\cF_1$ in the following way. Either $G$ has one more additional edge compared to $R$ or two more edges. Consider the vertex $s$ in $R$ from the definition of $\cF_1$. If $G$ and $R$ differs by just only one edge $e$, then $e$ must be incident with $s$, but the other endpoint of $e$ could be any other vertex. If $G$ and $R$ differs by exactly two edges $e_1,~e_2$, then both of them have to be incident with $s$ but some further conditions must be satisfied. By the definition of $\cF_1$, the number of vertices from $V(R)\setminus\{w_k\}$ that has degree equal to $3$ is exactly $d(w_k)-1$. If $d(w_k)\ge 4$, then $k>1$, and $e_1=sw_k$ is mandatory but $e_2$ can be an arbitrary edge incident with $s$. If $d(w_k)=3$, then let us denote the two other vertices of degree 3 with $x$ and $y$. In this case either $s$ is adjacent to $w_k$ while $k>1$ and $e_2$ is arbitrarily chosen, or $s$ is adjacent to a neighbor of $x$ and a neighbor of $y$. If $x$ or $y$ belongs to $W$, for example $x=w_i$ then $s$ must be adjacent to $w_{i-1}$, if $y$ belongs to some cycle component then it does not matter which neighbor of $y$ is adjacent to $s$. Observe that this second option is available if $x$ and $y$ is not $w_1$ or $w_2$ because $s$ cannot have additional edges to $w_1$ or $s$ since $G$ is a simple graph. If $d(w_k)=2$, then let us denote the only vertex of degree 3 with $x$. Similarly to the previous case, either $s$ is adjacent to $w_k$ while $k>1$, or $s$ is adjacent to another neighbor of $x$ and $e_2$ is arbitrarily chosen in both cases. If $x=w_i$ then $e_1=sw_{i-1}$ is necessary but if $x$ belongs to the cycle component then it does not matter which neighbor of $x$ from the cycle is adjacent to $s$.

In this paper, our main result is the characterization of graphs satisfying $\DOM_t(G)=|V(G)|-1$. 

\begin{theorem}\label{t:main}
    Let $G\in \cC$ be a connected graph. Then $\DOM_t(G)=|V(G)|-1$ if and only if $G\in\cF_1\cup \cF_2\cup\cF_3$. If $G\in\cC$ is not connected but $\DOM_t(G)=|V(G)|-1$ then $G$ is a disjoint union of some cycles and a graph $G_0\in\cF_1\cup\cF_2\cup\cF_3$.
\end{theorem}

The structure of the paper is the following. In Section 2, we summarize those observations and previous results that will be used later on. Section 3 contains the proof of our main result. In Section 4, we will conclude by showing graphs $G$ such that $\DOM_t(G)$ and $\dom_t(G)$ is as far as possible. 

\section{Preliminaries} 

Our task is to characterize those graphs that satisfy the $\DOM_t(G)=|V(G)|-1$ condition. As it was pointed out in the Introduction, in order to investigate $\DOM_t(G)$ we should restrict the family of graphs in question to the class $\cC$, i.e. every connected component of $G$ contains a cycle. Anderson et al. in \cite{ADJK} characterized those graphs $G\in\cC$ such that $\DOM_t(G)=|V(G)|$.

\begin{result}\label{r:equality}
Let $G\in\cC$, then $\DOM_t(G)=|V(G)|$ if and only if $G$ is a disjoint union of cycles.   
\end{result}

If the graph $G\in\cC$ is not connected then the total domination number of each component is not effected by the chosen orientation of the other components because there are no edges between the components hence the orientable total domination number of $G$ can be computed from the orientable total domination number of its components.

\begin{proposition}\label{p:concomp}
Let $G\in \cC$, and suppose that $G$ has $k\ge1$ connected components $G_1,G_2,\dots,$ $G_k$. Then $\DOM_t(G)=\sum_{i=1}^{k} \DOM_t(G_i)$.
\end{proposition}

By Proposition \ref{p:concomp}, if $\DOM_t(G)=|V(G)|-1$ then all but one connected component must be just a cycle, and the last component should be a connected graph $G_k$ such that $\DOM_t(G_k)=|V(G_k)|-1$. Therefore we should concentrate on the characterization of the connected case. Suppose in the sequel that the graph $G\in \cC$ in question is connected but not a cycle thus $\DOM_t(G)\le|V(G)|-1$ by Result \ref{r:equality}.

For a connected graph $G\in\cC$, let us call a valid orientation of $G$ \emph{extremal} if $\gamma_t(G)=|V(G)|-1$. In other words, a connected graph $G\in\cC$ satisfies $\DOM_t(G)=|V(G)|-1$ if and only if $G$ admits an extremal orientation. In the remaining part of this section, we collect some observations and properties of extremal orientations.

\begin{observation}\label{o:1}
Since an extremal orientation $o$ is also a valid orientation, every in-degree have to be at least 1, i.e. $d^o_{-}(v)\ge 1$ for every $v\in V$.
\end{observation}

\begin{observation}\label{o:gen}
Suppose that $G$ admits an extremal orientation, denoted by $o$. Then there is no vertex subset $S\subseteq V$ such that $S$ dominates at least $|S|+2$ vertices, i.e. $|\cup_{v\in S}~N(v)|\ge |S|+2$.
\end{observation}
\begin{proof}
    Let us suppose on the contrary that there exists $S\subseteq V$ such that $S$ dominates at least $|S|+2$ vertices, i.e. $|\cup_{v\in S}~N(v)|\ge |S|+2$. Thus we need to dominate the remaining at most $|V(G)|-|S|-2$ vertices. Since each vertex has at least one in-neighbor because $o$ is a valid orientation, hence we can select at most $|V(G)|-|S|-2$ further vertices to extend $S$ such that we get a total dominating set of size at most $|V(G)|-2$, which is a contradiction.
\end{proof}

\begin{corollary}\label{c:2} 
In an extremal orientation $o$ every out-degree must be at most 2, i.e. $d^o_{+}(v)\le 2$ for every $v\in V$.
\end{corollary}

\begin{corollary}\label{c:3}
    In an extremal orientation $o$, if $u\ne v$ are two distinct vertices such that their out-degree is 2 then $|N^o_{+}(u)\cup N^o_{+}(v)|\le 3$ must hold. Equivalently $u$ and $v$ must have at least one common out-neighbor.
\end{corollary}

\begin{observation}\label{o:4} In an extremal orientation $o$ there can be at most one vertex $v$ such that $d^o_{+}(v)=0$.
\end{observation}
\begin{proof}
    A vertex $v$ whose out-degree is zero can not dominate any vertices, hence $v$ cannot be contained in any minimal total dominating set of that orientation. Since $\gamma_t(G)=|V(G)|-1$ must hold hence there can be at most one such vertex.
\end{proof}


\section{Proof of the main result} 


In this section, we characterize those connected graphs $G\in\cC$ that admits an extremal orientation. We will present two proofs for the main case of our characterization, one of them is a thorough case analysis, and the other one has a recursive flavor to it. We opted to include both proofs although we think that the second proof is a bit more elegant, but the characterized families of graphs can be immediately seen from the first proof. Thus we will be able to easily identify those graphs that satisfy the $\dom_t(G)<\DOM_t(G)=|V(G)|-1$ inequality.

\begin{proof}[Proof of Theorem \ref{t:main} (case analysis)]
By Proposition \ref{p:concomp}, we can assume that $G\in \cC$ is connected, and one can eventually extend the constructions of the characterization with some disjoint cycle components.

The proof is based on a case analysis. In the first case let us suppose that a connected graph $G\in\cC$ admits an extremal orientation $o$ and there exists a vertex $s$ that has 0 out-neighbor, i.e. $d^o_{+}(s)=0$. 

Since $s$ cannot be included in any minimal total dominating set hence in the extremal orientation $o$, the set $V\setminus\{s\}$ must be the only minimal total dominating set. We need to understand the reasons why can't we omit another vertex from $V\setminus\{s\}$ to get a smaller total dominating set with respect to $o$.

If $d^o_{-}(s)\ge 2$, then consider the directed graph, denoted by $(G-s,o)$, that we get by deleting $s$ from $G$ and restrict $o$ to $V\setminus \{s\}$. We show that $\gamma_t(G-s,o)=|V(G-s,o)|=|V(G)|-1$ must hold. 
If there exists a total dominating set in $(G-s,o)$ with at most $|V(G)|-2$ vertices then by adding arbitrary vertices to this total dominating set $S$ we can get a total dominating set of size exactly $|V(G)|-2=|S|$. But $d^o_{-}(s)\ge 2$ implies that $N^o_{-}(s) \cap S\ne \emptyset$ by the pigeonhole principle, which means that $S$ is a total dominating set in $G$ with respect to $o$, too. That contradicts the fact that $o$ is an extremal orientation.

Hence $\gamma_t(G-s,o)=|V(G-s,o)|=|V(G)|-1$ holds and consequently $G-s$ must be a disjoint union of cycles according to Result \ref{r:equality}. Notice here that it might happen that $s$ is a cut vertex in the connected graph $G$. What can we say about those edges that are incident with $s$? 

Since there are no out-neighbors of $s$, and we assumed that $d^o_{-}(s)\ge 2$, there must be at least 2 edges incident with $s$ in $G$. Moreover, since $G$ is connected hence $s$ must be connected to every connected component of $G-s$ thus $\max \{2,\ell\}\le d(s)\le |V(G)|-1$ where $\ell$ is the number of connected components of $G-s$. Hence the graph $G$ can have a few cycles and a vertex $s$ that is connected to every such cycles with at least one edge.

\begin{figure}[H]
    \centering
    \begin{tikzpicture}[thick, main/.style = {draw, circle, fill, minimum size=6pt,inner sep=1pt, outer sep=1pt}]
\node[main] (s) {};
\node[above=.2cm] at (s) {$s$};
\node[main] (a) [above right = 2cm of s] {}; 
\node[main] (b) [right = 1cm of a] {}; 
\node[main] (c) at (2.2, 2.7) {};
\node[main] (d) [below right = 2cm of s]{}; 
\node[main] (e) [below = 1 cm of d]{}; 
\node[main] (g) [right = 1 cm of d] {}; 
\node[main] (f) [below = 1cm of g] {};
\node[main] (h) [left = 4cm of s] {};
\node[main] (i) [below = 1cm of h] {};
\node[main] (j) [right = 1cm of i] {};
\node[main] (k) [above = 1cm of j] {};
\draw[->] (a) -- (b);
\draw[->] (b) -- (c);
\draw[->] (c) -- (a);
\draw[->] (d) -- (e);
\draw[->] (e) -- (f);
\draw[->] (f) -- (g);
\draw[->] (g) -- (d);
\draw[->] (h) -- (i);
\draw[->] (i) -- (j);
\draw[->] (j) -- (k);
\draw[->] (k) -- (h);
\draw[->] (a) -- (s);
\draw[->] (k) -- (s);
\draw[->] (j) -- (s);
\draw[->] (g) -- (s);
\draw[->] (e) -- (s);
\draw[->] (d) -- (s);
\end{tikzpicture}   
    \caption{An extremal orientation of a graph with $d^o_+(s)=0$ and $d^o_-(s)\ge 2$}
    \label{fig:1}
\end{figure}
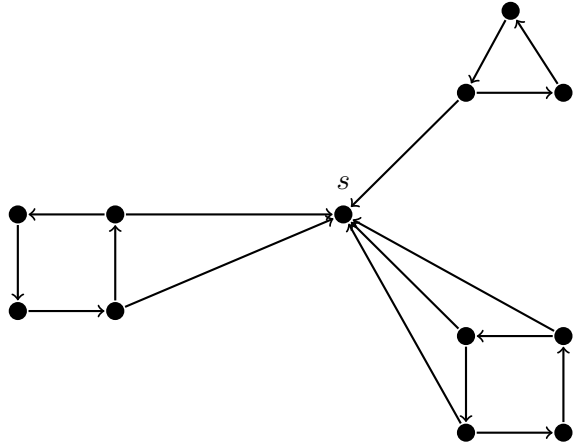

If the vertex $s$ has exactly one in-neighbor, denoted by $w_1$, then $w_1$ is the only vertex dominating $s$. Thus $w_1$ will certainly be in every total dominating set. If $w_1$ also has exactly one in-neighbor $w_2$, then by a similar argument the vertex $w_2$ must belong to any total dominating set. If we repeat this step then is it true that we eventually find a vertex $w_k$ among these $w_i$'s (for $i\ge 1$) such that it has at least two in-neighbors?

Since $G$ is a finite graph, the only case when there is no such vertex is if these steps have taken us back to a vertex $w_i$ already on the path (for some $1\le i< k-1$). But then these vertices form a directed cycle and a directed path from it. The subgraph induced by these vertices cannot have any more edges, because we have already included all the in-neighbors of these vertices. On the other hand, there are no edges going out from this $W\cup\{s\}$ part of $G$ because then $W$ would violate the condition in Observation \ref{o:gen}, since $W$ dominates all vertices of $W\cup\{s\}$ and then some if there would be outgoing edges from any vertex of $W$. Hence in this case the graph $G$ would belong to the family $\cF$ defined earlier.

\begin{figure}[H]
    \centering
\begin{tikzpicture}[thick, main/.style = {draw, circle, fill, minimum size=6pt,inner sep=1pt, outer sep=1pt}]
\node[main] (s) {};
\node[above] at (s) {$s$};
\node[main] (1) [left = 1.5cm of s] {};
\node[above] at (1) {$w_1$};
\node[main] (2) [left = 1.5cm of 1] {};
\node[above] at (2) {$w_2$};
\node[main] (3) [left = 1.5cm of 2] {};
\node[above] at (3) {$w_{i-1}$};
\node[main] (4) [left = 1.5cm of 3] {};
\node[above] at (4) {$w_i$};
\node[main] (5) [left = 1.5cm of 4] {};
\node[above] at (5) {$w_{i+1}$};
\node[main] (6) [left = 1.5cm of 5] {};
\node[above] at (6) {$w_{k-1}$};
\node[main] (7) [left = 1.5cm of 6] {};
\node[above] at (7) {$w_{k}$};
\draw[->] (1) -- (s);
\draw[->] (2) -- (1);
\node at (-4.45, 0) {$\dots$};
\draw[->] (4) -- (3);
\draw[->] (5) -- (4);
\node at (-9.8, 0) {$\dots$};
\draw[->] (7) -- (6);
\draw[->] (4) to [out=270,in=270,looseness=0.3] (7);
\draw[thick, dashed] (-7.2, 0) ellipse (6cm and 1.5cm);
\node at (-5, 2) {$W$};

\end{tikzpicture}
    \caption{An extremal orientation of a graph with $d^o_+(s)=0$ and $d^o_-(s)=1$, if there is no vertex that has at least 2 in-neighbors}
    \label{fig:2}
\end{figure}
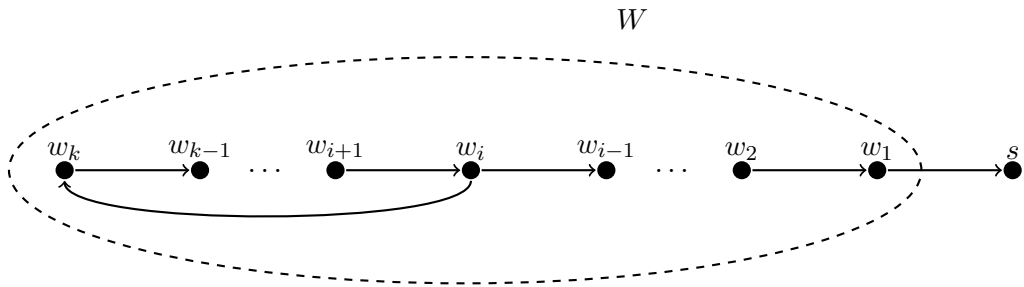

Let us suppose now that there exists a first vertex $w_k$ such that $d^o_-(w_k)\ge 2$ and there is a directed path $w_k w_{k-1} \dots w_2 w_1 s$ from $w_k$ to $s$ and $d^o_-(w_i)=1=d^o_-(s)$ holds for $1\le i< k$. Let $W=\{w_1,w_2,\dots,w_k\}$. 

If every in-neighbor of $w_k$ belongs to $W$ then there can be no further vertices of $G$. Suppose on the contrary that $\emptyset \ne V\setminus (W\cup \{s\})$. Since $G$ is connected there must be an edge $e$ between $W\cup \{s\}$ and $V\setminus (W\cup \{s\})$. But $e$ cannot be directed towards $W\cup\{s\}$ in the extremal orientation $o$ because we have already counted every incoming edge for these vertices. On the other hand, if $e$ is directed towards $V\setminus (W\cup \{s\})$ then $W$ would violate the condition in Observation \ref{o:gen} again. Hence $G$ is just a path from $w_k$ to $s$ and $w_k$ might be adjacent to any other $w_i$, too.

\begin{figure}[H]
    \centering
    \begin{tikzpicture}[thick, main/.style = {draw, circle, fill, minimum size=6pt,inner sep=1pt, outer sep=1pt}]
\node[main] (s) {};
\node[above] at (s) {$s$};
\node[main] (w) [left = 1.5cm of s] {};
\node[above] at (w) {$w_1$};
\node[main] (x) [left = 1.5cm of w] {};
\node[above] at (x) {$w_2$};
\node[main] (m) [left = 1.5cm of x] {};
\node[above] at (m) {$w_{k-2}$};
\node[main] (k) [left = 1.5cm of m] {};
\node[above] at (k) {$w_{k-1}$};
\node[main] (y) [left = 1.5cm of k] {};
\node[above] at (y) {$w_k$};
\node at (-4.5, 0) {$\dots$};
\draw[->] (w) -- (s);
\draw[->] (x) -- (w);
\draw[->] (k) -- (m);
\draw[->] (y) -- (k);
\draw[->] (w) to [out=270,in=270,looseness=0.5] (y);
\draw[->] (m) to [out=270,in=270,looseness=0.5] (y);
\draw[thick, dashed] (-5.3, 0) ellipse (4.5cm and 1.5cm);
\node at (-3, 2) {$W$};
\end{tikzpicture}
    \caption{An extremal orientation of a graph with $d^o_+(s)=0$ and $d^o_-(s)=1$, if there is a vertex $w_k$ that has at least 2 in-neighbors but all of these in-neighbors belong to $W$}
    \label{fig:3}
\end{figure}
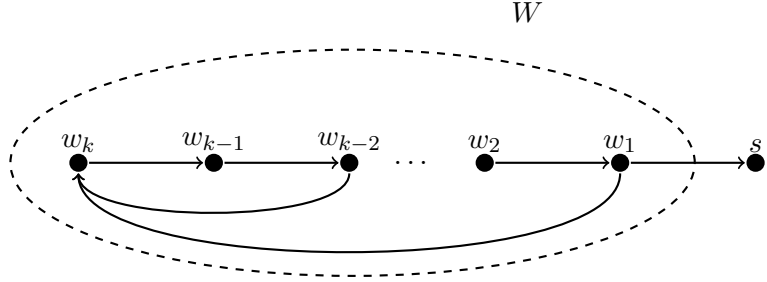

Suppose now that there is at least one in-neighbor of $w_k$ from both $W$ and $V\setminus(W\cup \{s\})$. Similarly to the previous cases, if there exists a directed edge coming out of $W$ towards $V\setminus(W\cup \{s\})$ then $W$ would violate the condition of Observation \ref{o:gen}. Since $V\setminus(W\cup \{s\})\ne\emptyset$, the directed graph $(G|_{V\setminus(W\cup \{s\})},o)$ restricted to $V\setminus(W\cup \{s\})$ must satisfy $\gamma_t(G|_{V\setminus(W\cup \{s\})},o)=|V(G\mid_{V\setminus(W\cup \{s\})},o)|=|V|-|W|-1$ because otherwise the union of $W$ and a minimal total dominating set in $(G|_{V\setminus(W\cup \{s\})},o)$ would be a total dominating set of $G$ such that the total size is at most $|V|-2$ which is a contradiction. Therefore $(G|_{V\setminus(W\cup \{s\})},o)$ must be a disjoint union of directed cycles and from each component there must be at least one directed edge going towards $W$, but more specifically to $w_k$ because the other vertices have just one in-neighbor in total.

\begin{figure}[H]
    \centering
    \begin{tikzpicture}[thick, main/.style = {draw, circle, fill, minimum size=6pt,inner sep=1pt, outer sep=1pt}]
\node[main] (s) {};
\node[above left] at (s) {$s$};
\node[main] (w) [below left = 1cm of  s] {};
\node[above left] at (w) {$w_1$};
\node[main] (w2) [below left of = w] {};
\node[above left] at (w2) {$w_2$};
\node[main] (x) [below left of = w2] {};
\node[above left] at (x) {$w_{k-2}$};
\node[main] (x2) [below left of = x] {};
\node[above left] at (x2) {$w_{k-1}$};
\node[main] (k) [below left of = x2] {};
\node[below=.2cm] at (k) {$w_k$};
\node[main] (f) [left = 2cm of k] {};
\node[main] (y) [left of = f] {};
\node[main] (a) [above of = f] {};
\node[main] (z) [left of = a]{};
\node[main] (h) [below right = 1cm of s] {};
\node[main] (i) [right of = h] {};
\node[main] (j) [below of = i] {};
\node[main] (l) [left of = j] {};
\draw[->] (w) -- (s);
\draw[->] (w2) -- (w);
\node[rotate=45] at (-1.92, -1.92) {$\dots$};
\draw[->] (k) -- (x2);
\draw[->] (x2) -- (x);
\draw[->] (y) -- (f);
\draw[->] (f) -- (k);
\draw[->] (z) -- (y);
\draw[->] (a) -- (z);
\draw[->] (f) -- (a);
\draw[->] (h) -- (i);
\draw[->] (i) -- (j);
\draw[->] (j) -- (l);
\draw[->] (l) -- (h);
\draw[->] (a) -- (k);
\draw[->] (w) to [out=270,in=360,looseness=0.7] (k);
\draw[->] (l) to [out=270,in=330,looseness=0.7] (k);
\draw[thick, dashed, rotate=45] (-3.3, 0) ellipse (2.7cm and 1.3cm);
\node at (-3, -0.5) {$W$};
\end{tikzpicture}
    \caption{An extremal orientation of a graph with $d^o_+(s)=0$ and $d^o_-(s)=1$, if there is a vertex $w_k$ that has at least 2 in-neighbors but $N^o_-(w_k)\cap W\ne\emptyset\ne N^o_-(w_k)\cap(V\setminus(W\cup \{s\}))$.}
    \label{fig:4}
\end{figure}
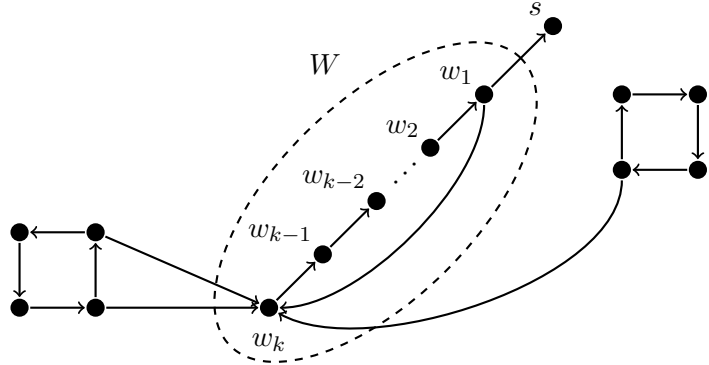

Suppose now that $d^o_-(w_k)\ge 2$ but $N^o_-(w_k)\cap W=\emptyset$, in other words all the in-neighbors of $w_k$ belong to $V\setminus(W\cup\{s\})$. Observation \ref{o:gen} implies that $\gamma_t(G|_{V\setminus(W\cup\{s\})})\ge |V\setminus(W\cup\{s\})|-1$. But if there exists a total dominating set $S$ in $G|_{V\setminus(W\cup\{s\})}$ such that $|S|=|V\setminus(W\cup\{s\})|-1$ then $W\cup S$ would again violate the condition in Observation \ref{o:gen} since $w_k$ is also dominated because $d^o_-(w_k)\ge 2$ and therefore $N^o_-(w_k)\cap S\ne\emptyset$.

Hence $\gamma_t(G|_{V\setminus(W\cup\{s\})})=|V\setminus(W\cup\{s\})|$ must hold, and consequently $G|_{V\setminus(W\cup\{s\})}$ have to be the disjoint union of some directed cycles, but there must be at least one directed edge to $w_k$ from each components, and at least two in-coming directed edges in total to $w_k$.

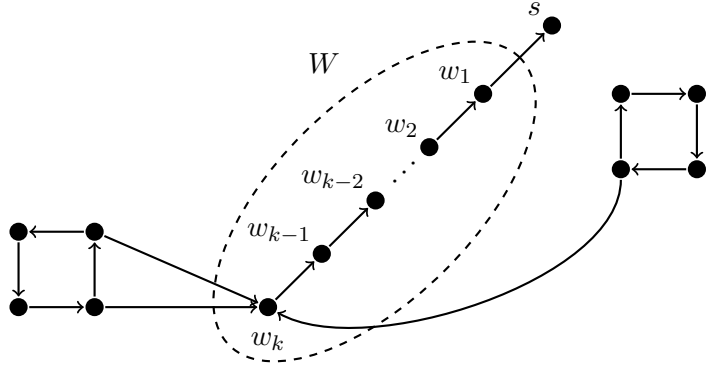
\begin{figure}[H]
    \centering
    \begin{tikzpicture}[thick, main/.style = {draw, circle, fill, minimum size=6pt,inner sep=1pt, outer sep=1pt}]
\node[main] (s) {};
\node[above left] at (s) {$s$};
\node[main] (w) [below left = 1cm of  s] {};
\node[above left] at (w) {$w_1$};
\node[main] (w2) [below left of = w] {};
\node[above left] at (w2) {$w_2$};
\node[main] (x) [below left of = w2] {};
\node[above left] at (x) {$w_{k-2}$};
\node[main] (x2) [below left of = x] {};
\node[above left] at (x2) {$w_{k-1}$};
\node[main] (k) [below left of = x2] {};
\node[below=.2cm] at (k) {$w_k$};
\node[main] (f) [left = 2cm of k] {};
\node[main] (y) [left of = f] {};
\node[main] (a) [above of = f] {};
\node[main] (z) [left of = a]{};
\node[main] (h) [below right = 1cm of s] {};
\node[main] (i) [right of = h] {};
\node[main] (j) [below of = i] {};
\node[main] (l) [left of = j] {};
\draw[->] (w) -- (s);
\draw[->] (w2) -- (w);
\node[rotate=45] at (-1.92, -1.92) {$\dots$};
\draw[->] (x2) -- (x);
\draw[->] (k) -- (x2);
\draw[->] (y) -- (f);
\draw[->] (f) -- (k);
\draw[->] (z) -- (y);
\draw[->] (a) -- (z);
\draw[->] (f) -- (a);
\draw[->] (h) -- (i);
\draw[->] (i) -- (j);
\draw[->] (j) -- (l);
\draw[->] (l) -- (h);
\draw[->] (a) -- (k);
\draw[->] (l) to [out=270,in=330,looseness=0.7] (k);
\draw[thick, dashed, rotate=45] (-3.3, 0) ellipse (2.7cm and 1.3cm);
\node at (-3, -0.5) {$W$};
\end{tikzpicture}
    \caption{An extremal orientation of a graph with $d^o_+(s)=0$ and $d^o_-(s)=1$, if there is a vertex $w_k$ that has at least 2 in-neighbors but $N^o_-(w_k)\cap W=\emptyset$.}
    \label{fig:5}
\end{figure}

Let us summarize what are those graph families for which there exists an extremal orientation $o$ such that there also exists a vertex $s$ with $d^o_+(s)=0$. If the minimum degree $\delta(G)\ge 2$ then $G$ must contain $\ell$ vertex-disjoint cycles (for some $\ell\ge 1$) and one additional vertex that is adjacent to at least one vertex from each cycle component (see Figure \ref{fig:1}), thus $G\in\cF_2$.

If $\delta(G)=1$ then there must be exactly one vertex of degree $1$, and let us call this vertex $s$. Then $G$ must contain a path of length $k$ from $s$, denoted by $sw_1 w_2 \dots w_{k-1} w_k$. If there are no further vertices in $G$ then there must exist at least one additional edge in $G$ but all of the additional edges must be incident with $w_k$, the other end of the path (see Figure \ref{fig:2} and \ref{fig:3}), thus $G\in \cF_1$ (where $\ell=0$ in the definition of $\cF_1$).

On the other hand, if there are more vertices in $G$, i.e. $V\setminus\{w_1,w_2,\dots,w_k,s\}\ne\emptyset$, then the rest must contain $\ell$ vertex-disjoint cycles (for some $\ell\ge 1$) and all of these cycle components must be connected to $w_k$, the other end of the path. Moreover, every additional edge must be incident with $w_k$, too (see Figure \ref{fig:4} and \ref{fig:5}), thus $G\in \cF_1$. 


\bigskip

\noindent{\it 2nd proof of the case when $d^o_+(s)=0$ of Theorem \ref{t:main} (recursive approach).}
Let us suppose again, that $G\in \cC$ is connected because one can eventually extend the constructions of the characterization with some disjoint cycle components by Proposition \ref{p:concomp}.

As in the first solution, let us suppose that $G$ admits an extremal orientation $o$ such that there exists a vertex $s$ that has not got an out-neighbor, i.e. $d^o_{+}(s)=0$. Since $s$ cannot be included in any minimal total dominating set hence the set $V\setminus\{s\}$ must be the only minimal total dominating set with respect to $o$. Why can't we find a smaller total dominating set? The reason is that each vertex $v\in V\setminus \{s\}$ must dominate a vertex $u_v\in N^o_+(v)$ such that no other vertex dominates $u$, i.e. $u\notin \cup_{w\in V\setminus\{s\}, v\ne w} N^o_+(w)$. 

Let us consider the directed graph, denoted by $(G-s,o)$, that we get by deleting $s$ from $G$ and restrict $o$ to $V\setminus \{s\}$. 
Then either $\gamma_t(G-s,o) = |V(G-s,o)|=|V(G,o)|-1$ or $\gamma_t(G-s,o) < |V(G-s,o)|=|V(G,o)|-1$. If $\gamma_t(G-s,o)=|V(G-s,o)|=|V(G,o)|-1$ holds, consequently $G-s$ must be a disjoint union of cycles according to Result \ref{r:equality}. Thus $G\in\cF_2$ (see Figure \ref{fig:1}). 

Suppose now that $\gamma_t(G-s,o) < |V(G-s,o)|=|V(G,o)|-1$. We get a contradiction if $\gamma_t(G-s,o) < |V(G,o)|-2$, because if we add one of the in-neighbors of $s$ to the minimal total dominating set in $(G-s,o)$, we get a total dominating set in $G$ with size at most $|V(G,o)|-2$. Hence $\gamma_t(G-s,o) = |V(G,o)|-2=|V(G-s,o)|-1$ holds.

How many in-neighbors can $s$ have? 
If $d^o_{-}(s)\ge 2$, then the total dominating set of size $|V(G,o)-2|$ in $(G-s,o)$ is also a total dominating set in $G$ by the pigeonhole principle, thus it contradicts that $o$ is an extremal orientation. Together with Observation \ref{o:1}, it yields that $d^o_{-}(s) = 1$. Let us denote the only in-neighbor of $s$ with $w_1$, hence $1\le d^o_+(w_1)\le 2$ by Corollary \ref{c:2}.

If $d^o_+(w_1)=1$, then it implies, that in $(G-s,o)$ we get an analogue situation but now $w_1$ plays the role of $s$: we have an extremal orientation $o$ in $(G-s,o)$ and a vertex $w_1$ such that $d^{(G-s,o)}_+(w_1)=0$, and $\gamma_t(G-s,o) = |V(G,o)|-2 = |V(G-s,o)|-1$. Again, if $\gamma_t(G-s-w_1,o) = |V(G,o)|-2 = |V(G-s,o)|-1$ then $G|_{V\setminus\{s,w_1\}}$ is a vertex-disjoint union of cycles by Result \ref{r:equality}, and every cycle component must be connected to $w_1$ thus $G\in \cF_1$. On the other hand, if $\gamma_t(G-s-w_1,o) < |V(G,o)|-2 = |V(G-s,o)|-1$ then by the same argument $\gamma_t(G-s-w_1,o) = |V(G,o)|-3 = |V(G-s,o)|-2$ and $d^{(G-s,o)}_-(w_1)=1$ is necessary because otherwise it would contradict Observation \ref{o:gen}. Let us denote the only in-neighbor of $w_1$ with $w_2$, and therefore $1\le d^o_+(w_2)\le 2$ by Corollary \ref{c:2}. If $d^o_+(w_2)=1$, then we can repeat the same argument and either get a graph from $\cF_1$ or deduce that there exists a vertex $w_3$ with similar properties as $w_1$ and $w_2$ before. We can iterate this step finitely many times, and that is the reason why we call this 2nd proof recursive.

Suppose that eventually we get a vertex $w_i$ such that $d^o_+(w_i)=2$, and denote the set $\{w_1,w_2,\dots,w_i\}$ with $W_i$. Let us denote the graph that we get by restricting the orientation $o$ to $V\setminus\{s,w_1,w_2,\dots,w_{i-1}\}$ by $(G_i,o)$. Then $d^{(G_i,o)}_+(w_i)=1$, and let us denote the only out-neighbor of $w_i$ in $(G_i,o)$ with $t_1$. By Observation \ref{o:gen}, we also know that $\gamma_t(G_i,o)=|V(G_i)|-1$ and $d^{(G_i,o)}_-(w_i)=1$. Moreover the only minimal total dominating set $S_i$ in $(G_i,o)$ must be $S_i=V(G_i)\setminus\{w_i\}=V\setminus(W_i\cup\{s\})$ because otherwise $S_i\cup W_i$ would be a total dominating set in $G$ of size $|V(G)|-2$ which is a contradiction. 

We need to understand, why can't we substitute a vertex from $S_i$ with $w_i$ and get a total dominating set of the same size. As it was pointed out earlier, the reason is that for any vertex $v\in S_i$ there exists a vertex $u_v\in N^{(G_i,o)}_+(v)$ such that $v$ is the only in-neighbor of $u_v$, thus $v$ must belong to any total dominating set. Consider $t_1$ and let us denote the vertex that is dominated only by $t_1$ with $t_2$. Then consider the vertex $t_3$ that is only dominated by $t_2$, etc. How can this process stop? Since the graph is finite, there is only one way to stop when the next $t_j$ coincide with $w_i$ (the only vertex outside of $S_i$). Note that these vertices form a directed cycle. 

It might happen that $V(G_i)\setminus\{t_1,t_2,\dots,t_{j-1},w_i\}\ne\emptyset$ then consider an arbitrary vertex $v$ from $V(G_i)\setminus\{t_1,t_2,\dots,t_{j-1},w_i\}$ and start a similar process from $v$ and always move to a new vertex that was dominated only by the previous vertex. This process would stop again by arriving back to $v$. We repeat this process until every vertex of $G_i$ is covered by one of these directed cycles. This way we determined a subgraph $(H_i,o)$ of $(G_i,o)$ that contains vertex-disjoint directed cycles. Since $S_i$ is a total dominating set in $(G_i,o)$ there must exists a directed edge from a vertex of $S_i$ to $t_1$. Observe that even in $(H_i,o)$ there exists an outgoing edge from every vertex, thus By Corollary \ref{c:2} and \ref{c:3}, we know that those vertices that has two out-neighbors must have a common out-neighbor with $w_i$ which has also two out-neighbors: $w_{i-1}$ and $t_1$. But $w_i$ was the only in-neighbor of $w_{i-1}$ hence every vertex $u$ in $(G_i,o)$ such that $d^{(G_i,o)}_+(u)=d^o_+(u)=2$ must be adjacent to $t_1$ as well, and consequently $d^o_+(t_1)=1$ holds. There can be no further edges in $(G_i,o)$ hence we can deduce that the underlying graph must belong to $\cF_1$ (see Figure \ref{fig:6}), where $t_1$ plays the role of $w_k$ from the first proof, since there is a directed path $t_1 t_2 \dots t_{j-1} w_i w_{i-1}\dots w_1 s$ from $t_1$ to $s$.

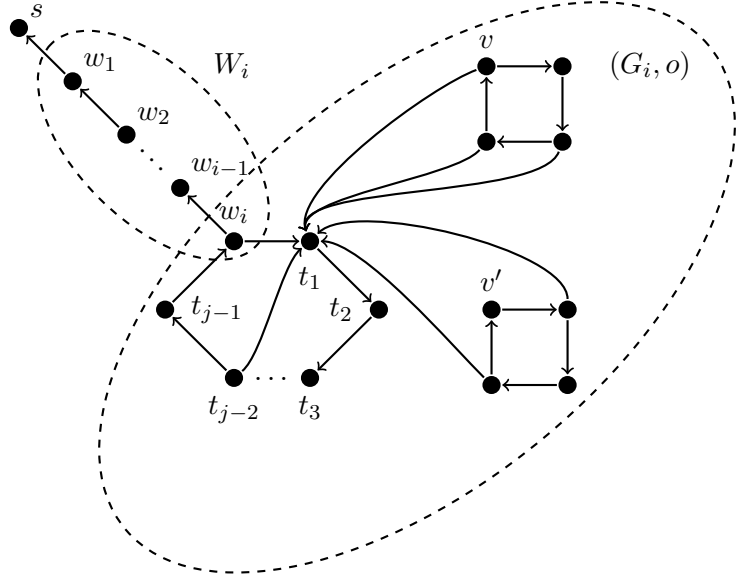
\begin{figure}[H]
    \centering
    \begin{tikzpicture}[thick, main/.style = {draw, circle, fill, minimum size=6pt,inner sep=1pt, outer sep=1pt}]
\node[main] (s) {};
\node[above right] at (s) {$s$};
\node[main] (w1) [below right of = s] {};
\node[above right] at (w1) {$w_1$};
\node[main] (w2) [below right of = w1] {};
\node[above right] at (w2) {$w_2$};
\node[main] (wk1) [below right of = w2] {};
\node[above right] at (wk1) {$w_{i-1}$};
\node[main] (wk) [below right of = wk1] {};
\node[above=.1cm] at (wk) {$w_{i}$};
\node[main] (t1) [right of = wk] {};
\node[below=.2cm] at (t1) {$t_1$};
\node[main] (u1) [below right = 1 cm of t1] {};
\node[left=.2cm] at (u1) {$t_{2}$};
\node[main] (u2) [below left = 1 cm of u1] {};
\node[below=.1cm] at (u2) {$t_{3}$};
\node[main] (ul1) [left of = u2] {};
\node[below=.1cm] at (ul1) {$t_{j-2}$};
\node[main] (ul) [above left = 1 cm of ul1]{};
\node[right=.2cm] at (ul) {$t_{j-1}$};

\node[main] (h) [above right = 3cm of t1] {};
\node[above=.1cm] at (h) {$v$};
\node[main] (i) [right of = h] {};
\node[main] (j) [below of = i] {};
\node[main] (l) [left of = j] {};

\node[main] (h1) [right = 1.2cm of u1] {};
\node[above=.1cm] at (h1) {$v'$};
\node[main] (i1) [right of = h1] {};
\node[main] (j1) [below of = i1] {};
\node[main] (l1) [left of = j1] {};

\draw[->] (w1) -- (s);
\draw[->] (w2) -- (w1);
\node[rotate=-45] at (1.77, -1.77) {$\dots$};
\draw[->] (wk) -- (wk1);
\draw[->] (wk) -- (t1);
\draw[->] (t1) -- (u1);
\draw[->] (u1) -- (u2);
\node at (3.35, -4.65) {$\dots$};
\draw[->] (ul1) -- (ul);
\draw[->] (ul) -- (wk);
\draw[->] (h) -- (i);
\draw[->] (i) -- (j);
\draw[->] (j) -- (l);
\draw[->] (l) -- (h);
\draw[->] (h1) -- (i1);
\draw[->] (i1) -- (j1);
\draw[->] (j1) -- (l1);
\draw[->] (l1) -- (h1);
\draw[->] (h) to [out=200,in=110,looseness=0.5] (t1);
\draw[->] (ul1) to [out=40,in=220,looseness=0.5] (t1);
\draw[->] (l) to [out=230,in=110,looseness=0.5] (t1);
\draw[->] (j) to [out=250,in=110,looseness=0.5] (t1);
\draw[->] (l1) to [out=135,in=0,looseness=0.5] (t1);
\draw[->] (i1) to [out=90,in=45,looseness=0.5] (t1);
\draw[thick, dashed, rotate=135] (-2.35, -0.15) ellipse (1.9cm and 1cm);
\node at (2.8, -0.5) {$W_i$};

\draw[thick, dashed, rotate=40] (1.8, -6) ellipse (5cm and 2.6cm);
\node at (8.3, -0.5) {$(G_i,o)$};
\end{tikzpicture}
    \caption{After a few recursive steps, the process arrived at $w_i$ such that $d^o_+(w_i)=2$, and the only out-neighbor of $w_i$ in $(G_i,o)$ is $t_1$.}
    \label{fig:6}
\end{figure}

\bigskip


In the prequel, we characterized those graphs in two different ways that admit an extremal orientation $o$ such that there exists a vertex $s$ with $d^o_+(s)=0$. Let's focus now on those graphs $G$ such that $\DOM_t(G)=|V(G)|-1$ but in the extremal orientation $o$ every out-degree is positive. Suppose that $S\subseteq V$ is a total dominating set in $G$ with respect to $o$, such that $|S|=|V(G)|-1$ and denote the only element in $V\setminus S$ with $s$. Consider the directed graph, denoted by $(R,o_R)$, that we get by deleting the outgoing edges of $s$ from $G$ and restricting $o$ to $R$. By Corollary \ref{c:2}, we know that $1\le d^o_+(s) \le 2$.

Observe that $o_R$ is an extremal orientation of $(R,o_R)$, because otherwise there would exists a total dominating set $T$ of size at most $|V(R)|-2=|V(G)|-2$ in $(R,o_R)$ such that $s\notin T$ because $s$ cannot dominate any vertex, thus the same $T$ would have been a total dominating set in $(G,o)$ as well, which is a contradiction.

From the first part of the proof, we can deduce that $R$ must be a graph from $\cF_1\cup\cF_2$. Since we get $R$ from $G$ by deleting the outgoing edges of $s$, we know that all the other graphs that admit an extremal orientation can be constructed from a member of $\cF_1\cup\cF_2$ by adding 1 or 2 outgoing edges of the special vertex $s$, which had 0 out-neighbor originally. The remaining task is to decide what vertices could be the other endpoints of the additional edges.

If there is just 1 edge that was deleted from $G$ then we claim that the other endpoint could be any vertex from $V\setminus\{s\}$. Let us assume that the only additional edge in $G$ leads from $s$ to an arbitrary vertex $v$. Suppose on the contrary that there would be a total dominating set $T$ of size $|V(G)|-2$ in the resulting graph $G$. Then $T$ has to contain $s$ because otherwise $T$ would be a total dominating set of size $|V(G)|-2$ in $(R,o_R)$ which is a contradiction. But $S=V\setminus\{s\}$ is also a total dominating set in $(R,o_R)$ therefore there exists a vertex $u\in S$ such that $u$ dominates $v$ in $(R,o_R)$. This gives rise to a contradiction since then $(T\setminus\{s\})\cup\{u\}$ would definitely be a total dominating set of size $|V(G)|-2$ in $(R,o_R)$. Hence we can obtain new examples by adding just 1 outgoing edge from $s$ to an arbitrary other vertex.

Suppose now that 2 outgoing edges were deleted from $G$. By Corollary \ref{c:3}, for any vertex $v\in V\setminus\{s\}$ such that $d^o_+(v)=2$, the intersection $N^o_+(s)\cap N^o_+(v)$ must be non-empty. 
If the remaining graph $R\in\cF_2$ then there is at least two vertices with out-degree 2 and $s$ is the common out-neighbor of all such vertices. Therefore if $|\{v~|~d^o_+(v)=2,~v\ne s\}|\ge 3$ then $N^o_+(s)$ cannot intersect each out-neighborhood because $s\notin N^o_+(s)$. If $|\{v~|~d^o_+(v)=2,~v\ne s\}|=2$ then let us denote these two vertices with $v_1$ and $v_2$. Then the only possibility is to add two outgoing edges from $s$ to the two different out-neighbors of $v_1$ and $v_2$. Let us note here that the resulting undirected graph $G$ still belongs to the $\cF_2$ family thus we actually did not find a new example this way.

If the remaining graph $R\in\cF_1$, then every vertex with out-degree 2 must dominate $w_k$. Again by Corollary \ref{c:3}, if $|\{v~|~d^o_+(v)=2,~v\ne s\}|\ge 3$ then $w_k\in N^o_+(s)$ is mandatory. But we claim that the other out-neighbor of $s$ could be any vertex in $V\setminus\{s,w_k\}$. Suppose on the contrary that $N^o_+(s)=\{v,w_k\}$ for some vertex $v$, and there exists a total dominating set $T$ with respect to $(G,o)$ such that $|T|=|V(G)|-2$. By the pigeonhole principle there must exists another vertex in $T$ besides $s$ that dominates $w_k$, and every vertex in $V\setminus\{w_k\}$ has exactly one in-neighbor hence we can replace $s$ with the only in-neighbor of $v$ in $(R,o_R)$ and get a total dominating set with respect to $(R,o_R)$ of the same size which is a contradiction. This way we obtained some new undirected graphs $G$ that satisfies $\DOM_t(G)=|V(G)|-1$ but does not admit an extremal orientation with a vertex of 0 out-degree.

If $R\in\cF_1$, but $|\{v~|~d^o_+(v)=2,~v\ne s\}|=1$ then actually $R\in\cF$ (see Figure \ref{fig:2}). Let us assume that $w_i$ is the only vertex besides $s$ such that $d^o_+(w_i)=2$. If $N^o_+(w_i)\cap N^o_+(s)\ne\emptyset$ then the resulting graph $G$ satisfies $\DOM_t(G)=|V(G)|-1$. Indeed, since every vertex $v\in V\setminus \{s,w_i\}$ could dominate exactly one vertex, and $\{s,w_i\}$ could dominate at most 3, thus any total dominating set must have size at least $|V(G)|-1$. Finally if $R\in\cF_1$, but $|\{v~|~d^o_+(v)=2,~v\ne s\}|=2$ then let us denote these two vertices with $v_1$ and $v_2$. In this case either $s$ dominates $w_k$, and $w_k$ is the common out-neighbor of all three of them, or $s$ dominates the other two out-neighbors of $v_1$ and $v_2$. Similarly to the previous cases we can use Corollary \ref{c:3} to prove that these will indeed be extremal orientations of $G$.

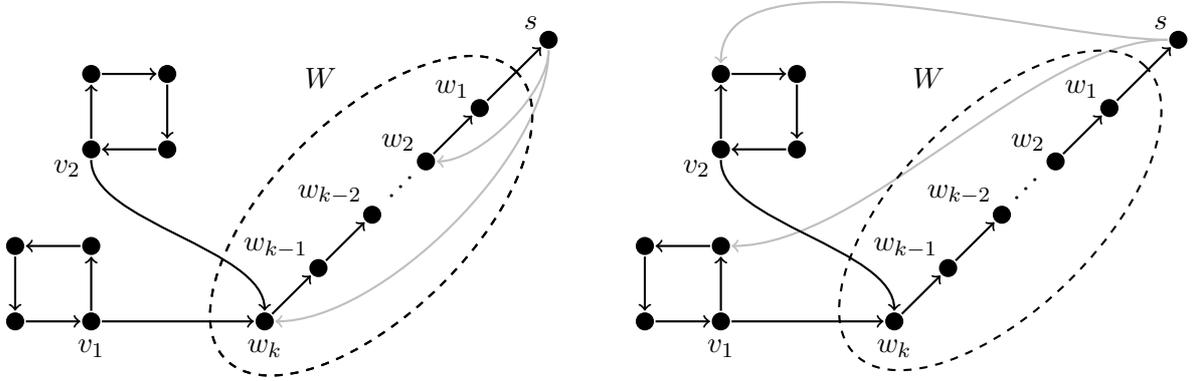
\begin{figure}[H]
    \centering
    \begin{tikzpicture}[thick, main/.style = {draw, circle, fill, minimum size=6pt,inner sep=1pt, outer sep=1pt}]
\node[main] (s) {};
\node[above left] at (s) {$s$};
\node[main] (w) [below left = 1cm of  s] {};
\node[above left] at (w) {$w_1$};
\node[main] (w2) [below left of = w] {};
\node[above left] at (w2) {$w_2$};
\node[main] (x) [below left of = w2] {};
\node[above left] at (x) {$w_{k-2}$};
\node[main] (x2) [below left of = x] {};
\node[above left] at (x2) {$w_{k-1}$};
\node[main] (k) [below left of = x2] {};
\node[below=.1cm] at (k) {$w_k$};
\node[main] (f) [left = 2cm of k] {};
\node[below=.1cm] at (f) {$v_1$};
\node[main] (y) [left of = f] {};
\node[main] (a) [above of = f] {};
\node[main] (z) [left of = a]{};
\node[main] (h) [above = 2cm of a] {};
\node[main] (i) [right of = h] {};
\node[main] (j) [below of = i] {};
\node[main] (l) [left of = j] {};
\node[below left ] at (l) {$v_2$};
\draw[->] (w) -- (s);
\draw[->] (w2) -- (w);
\node[rotate=45] at (-1.92, -1.92) {$\dots$};
\draw[->] (x2) -- (x);
\draw[->] (k) -- (x2);
\draw[->] (y) -- (f);
\draw[->] (f) -- (k);
\draw[->] (z) -- (y);
\draw[->] (a) -- (z);
\draw[->] (f) -- (a);
\draw[->] (h) -- (i);
\draw[->] (i) -- (j);
\draw[->] (j) -- (l);
\draw[->] (l) -- (h);
\draw[->, gray!50] (s) to [out=270,in=0,looseness=0.7] (k);
\draw[->, gray!50] (s) to [out=270,in=0,looseness=0.7] (w2);
\draw[->] (l) to [out=270,in=90,looseness=0.7] (k);
\draw[thick, dashed, rotate=45] (-3.3, 0) ellipse (2.7cm and 1.3cm);
\node at (-3, -0.5) {$W$};

\node[main] (s') [right = 8cm of s] {};
\node[above left] at (s') {$s$};
\node[main] (w') [below left = 1cm of  s'] {};
\node[above left] at (w') {$w_1$};
\node[main] (w2') [below left of = w'] {};
\node[above left] at (w2') {$w_2$};
\node[main] (x') [below left of = w2'] {};
\node[above left] at (x') {$w_{k-2}$};
\node[main] (x2') [below left of = x'] {};
\node[above left] at (x2') {$w_{k-1}$};
\node[main] (k') [below left of = x2'] {};
\node[below=.1cm] at (k') {$w_k$};
\node[main] (f') [left = 2cm of k'] {};
\node[below=.1cm] at (f') {$v_1$};
\node[main] (y') [left of = f'] {};
\node[main] (a') [above of = f'] {};
\node[main] (z') [left of = a']{};
\node[main] (h') [above = 2cm of a'] {};
\node[main] (i') [right of = h'] {};
\node[main] (j') [below of = i'] {};
\node[main] (l') [left of = j'] {};
\node[below left] at (l') {$v_2$};
\draw[->] (w') -- (s');
\draw[->] (w2') -- (w');
\node[rotate=45] at (6.3, -1.95) {$\dots$};
\draw[->] (x2') -- (x');
\draw[->] (k') -- (x2');
\draw[->] (y') -- (f');
\draw[->] (f') -- (k');
\draw[->] (z') -- (y');
\draw[->] (a') -- (z');
\draw[->] (f') -- (a');
\draw[->] (h') -- (i');
\draw[->] (i') -- (j');
\draw[->] (j') -- (l');
\draw[->] (l') -- (h');
\draw[->, gray!50] (s') to [out=180,in=90,looseness=0.7] (h');
\draw[->, gray!50] (s') to [out=180,in=0,looseness=0.7] (a');
\draw[->] (l') to [out=270,in=90,looseness=0.7] (k');
\draw[thick, dashed, rotate=45] (2.6, -5.8) ellipse (2.7cm and 1.3cm);

\draw[thick, dashed, rotate=45] (-3.3, 0) ellipse (2.7cm and 1.3cm);
\node at (5, -0.5) {$W$};

\end{tikzpicture}
    \caption{The same graph $R\in\cF_1$ for which $|\{v~|~d(v)=3,v\ne w_k\}|=2$ can be extended with two edges from $s$ in totally different ways such that both of the resulting graphs belong to $\cF_3$ and admit extremal orientations that coincide restricted to $R$.}
    \label{fig:7}
\end{figure}

We point out that those graphs $G$ that admit an extremal orientation but does not belong to $\cF_1\cup\cF_2$ can be constructed from a graph $R\in\cF_1$ by adding 1 or 2 further edges to the only vertex of degree 1 in $R$. Observe that the possibilities that were explained in the previous paragraphs mean that $G\in\cF_3$. 

To conclude the proof, we refer to Proposition \ref{p:concomp} and deduce that if $G$ is not connected but $\DOM_t(G)=|V(G)|-1$ then there must exists exactly one connected component $G'$ that also satisfies $\DOM_t(G')=|V(G')|-1$ and all other connected components must be cycles by Result \ref{r:equality}. But then $G'$ must belong to $\cF_1\cup\cF_2\cup\cF_3$ which finishes the proof. 
\end{proof}

\section{Conclusion, open problems} 

As it can be seen from the proof of our main result, every graph that satisfies $\dom_t(G)=|V(G)|-1$ also satisfies $\DOM_t(G)=|V(G)|-1$. Note that this is also a consequence of Result \ref{r:equality}, and the fact that $\dom_t(G)\le\DOM_t(G)$ by definition. But there are significantly more graphs that satisfy $\DOM_t(G)=|V(G)|-1>\dom_t(G)$. For these graphs, how small can $\dom_t(G)$ be?

It turned out that $\dom_t(G)$ can be as small as 3 for some graphs $G$ such that $\DOM_t(G)=|V(G)|-1$. Observe that $\dom_t(G)\ge 3$ always holds because $G$ is a simple graph. Indeed, one vertex cannot dominate itself (no loops), and two vertices cannot dominate each other (no parallel edges). Hence these examples are best possible for maximizing both $\frac{\DOM_t(G)}{\dom_t(G)}$ and $\DOM_t(G)-\dom_t(G)$. In Figure \ref{fig:8} and \ref{fig:9}, we show two such examples.

\begin{figure}[H]
    \centering
    \begin{tikzpicture}[thick, main/.style = {draw, circle, fill, minimum size=6pt,inner sep=1pt, outer sep=1pt}]
\node[main] (s) {};
\node[below=.2cm] at (s) {$s$};
\node[main, fill=gray!30, minimum size=0.35cm] (a) [above = 2cm of s] {}; 
\node[main, fill=gray!30, minimum size=0.35cm] (b) [above right = 1cm of a] {}; 
\node[main, fill=gray!30, minimum size=0.35cm] (c) [above left = 1cm of a] {}; 
\node[main, fill=gray!30, minimum size=0.35cm] (d) [right = 2cm of s]{}; 
\node[main, fill=gray!30, minimum size=0.35cm] (e) [below = 1 cm of d]{}; 
\node[main, fill=gray!30, minimum size=0.35cm] (g) [right = 1 cm of d] {}; 
\node[main, fill=gray!30, minimum size=0.35cm] (f) [below = 1cm of g] {};
\node[main, fill=gray!30, minimum size=0.35cm] (h) [left = 3cm of s] {};
\node[main, fill=gray!30, minimum size=0.35cm] (i) [below = 1cm of h] {};
\node[main, fill=gray!30, minimum size=0.35cm] (j) [right = 1cm of i] {};
\node[main, fill=gray!30, minimum size=0.35cm] (k) [above = 1cm of j] {};
\draw[->] (a) -- (b);
\draw[->] (b) -- (c);
\draw[->] (c) -- (a);
\draw[->] (d) -- (e);
\draw[->] (e) -- (f);
\draw[->] (f) -- (g);
\draw[->] (g) -- (d);
\draw[->] (h) -- (i);
\draw[->] (i) -- (j);
\draw[->] (j) -- (k);
\draw[->] (k) -- (h);

\draw[->] (a) -- (s);
\draw[->] (b) -- (s);
\draw[->] (c) -- (s);
\draw[->] (d) -- (s);
\draw[->] (e) -- (s);
\draw[->] (f) to [out=250,in=300,looseness=1] (s);
\draw[->] (g) to [out=120,in=40,looseness=1] (s);
\draw[->] (h) to [out=70,in=120,looseness=1] (s);
\draw[->] (i) to [out=290,in=240,looseness=1] (s);
\draw[->] (j) -- (s);
\draw[->] (k) -- (s);

\node[main, fill=gray!30, minimum size=0.35cm] (s') [right = 8cm of s] {}; 
\node[below=.2cm] at (s') {$s$};
\node[main] (a') [above = 2cm of s'] {}; 
\node[main] (b') [above right = 1cm of a'] {}; 
\node[main] (c') [above left = 1cm of a'] {}; 
\node[main] (d') [right = 2cm of s']{}; 
\node[main] (e') [below = 1 cm of d']{}; 
\node[main] (g') [right = 1 cm of d'] {}; 
\node[main] (f') [below = 1cm of g'] {};
\node[main] (h') [left = 3cm of s'] {};
\node[main] (i') [below = 1cm of h'] {};
\node[main, fill=gray!30, minimum size=0.35cm] (j') [right = 1cm of i'] {};
\node[main, fill=gray!30, minimum size=0.35cm] (k') [above = 1cm of j'] {};
\draw[->] (a') -- (b');
\draw[->] (b') -- (c');
\draw[->] (c') -- (a');
\draw[->] (d') -- (e');
\draw[->] (e') -- (f');
\draw[->] (f') -- (g');
\draw[->] (g') -- (d');
\draw[->] (h') -- (i');
\draw[->] (i') -- (j');
\draw[->] (j') -- (k');
\draw[->] (k') -- (h');

\draw[<-] (a') -- (s');
\draw[<-] (b') -- (s');
\draw[<-] (c') -- (s');
\draw[<-] (d') -- (s');
\draw[<-] (e') -- (s');
\draw[<-] (f') to [out=250,in=300,looseness=1] (s');
\draw[<-] (g') to [out=120,in=40,looseness=1] (s');
\draw[<-] (i') to [out=290,in=240,looseness=1] (s');
\draw[<-] (j') -- (s');
\draw[->] (k') -- (s');
\draw[<-] (h') to [out=70,in=130,looseness=1] (s');
\draw[thick, dotted] (7.3, -0.6) ellipse (1.6cm and 1.6cm);

\end{tikzpicture}   
    \caption{An example for a graph $G\in\cF_2$ with $3=\dom_t(G) < \DOM_t(G)=|V(G)|-1$. The gray nodes mark the minimal total dominating sets in both orientations of $G$.}
    \label{fig:8}
\end{figure}
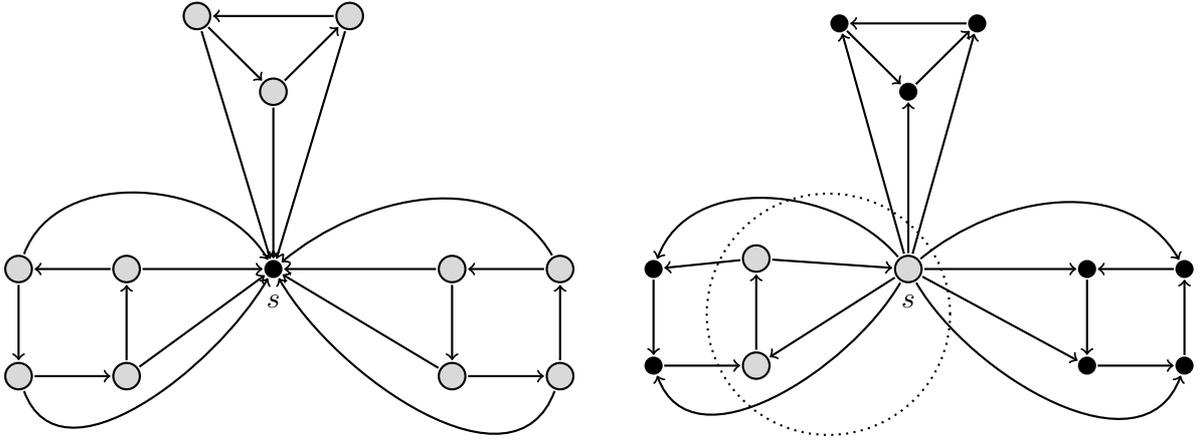

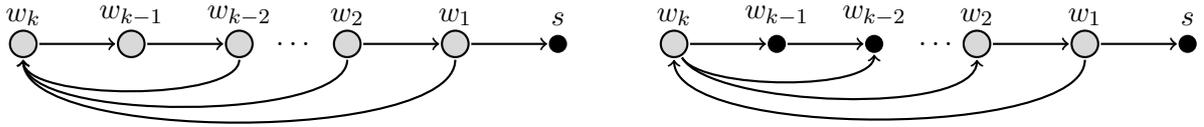
\begin{figure}[H]
    \centering
    \begin{tikzpicture}[thick, main/.style = {draw, circle, fill, minimum size=6pt,inner sep=1pt, outer sep=1pt}]
\node[main] (s) {};
\node[above = 0.1 cm] at (s) {$s$};
\node[main, fill=gray!30, minimum size=0.35cm] (w) [left = 1cm of s] {};
\node[above = 0.1 cm] at (w) {$w_1$};
\node[main, fill=gray!30, minimum size=0.35cm] (x) [left = 1cm of w] {};
\node[above = 0.1 cm] at (x) {$w_2$};
\node[main, fill=gray!30, minimum size=0.35cm] (m) [left = 1cm of x] {};
\node[above = 0.1 cm] at (m) {$w_{k-2}$};
\node[main, fill=gray!30, minimum size=0.35cm] (k) [left = 1cm of m] {};
\node[above = 0.1 cm] at (k) {$w_{k-1}$};
\node[main, fill=gray!30, minimum size=0.35cm] (y) [left = 1cm of k] {};
\node[above = 0.1 cm] at (y) {$w_k$};
\node at (-3.45, 0) {$\dots$};
\draw[->] (w) -- (s);
\draw[->] (x) -- (w);
\draw[->] (k) -- (m);
\draw[->] (y) -- (k);
\draw[->] (w) to [out=270,in=270,looseness=0.5] (y);
\draw[->] (x) to [out=270,in=270,looseness=0.5] (y);
\draw[->] (m) to [out=270,in=270,looseness=0.5] (y);

\node[main] (s') [right = 8cm of s] {};
\node[above = 0.1 cm] at (s') {$s$};
\node[main, fill=gray!30, minimum size=0.35cm] (w') [left = 1cm of s'] {};
\node[above = 0.1 cm] at (w') {$w_1$};
\node[main, fill=gray!30, minimum size=0.35cm] (x') [left = 1cm of w'] {};
\node[above = 0.1 cm] at (x') {$w_2$};
\node[main] (m') [left = 1cm of x'] {};
\node[above = 0.1 cm] at (m') {$w_{k-2}$};
\node[main] (k') [left = 1cm of m'] {};
\node[above = 0.1 cm] at (k') {$w_{k-1}$};
\node[main, fill=gray!30, minimum size=0.35cm] (y') [left = 1cm of k'] {};
\node[above = 0.1 cm] at (y') {$w_k$};
\node at (5, 0) {$\dots$};
\draw[->] (w') -- (s');
\draw[->] (x') -- (w');
\draw[->] (k') -- (m');
\draw[->] (y') -- (k');
\draw[->] (w') to [out=270,in=270,looseness=0.5] (y');
\draw[<-] (x') to [out=270,in=300,looseness=0.5] (y');
\draw[<-] (m') to [out=270,in=300,looseness=0.5] (y');

\end{tikzpicture}
    \caption{An example for a graph $G\in\cF_1$ with $3=\dom_t(G) < \DOM_t(G)=|V(G)|-1$. The gray nodes mark the minimal total dominating sets in both orientations of $G$.}
    \label{fig:9}
\end{figure}

In these examples the two orientations drastically differ on the edges incident with $w_k$. The larger the out-degree of $w_k$, the smaller $\dom_t(G)$ can get. Therefore there are examples similar to the one in Figure \ref{fig:9} such that $\dom_t(G)=m$ for any $3\le m \le |V(G)|-1$.

What can we say on the number of edges in graphs that admit extremal orientations? From Result \ref{r:equality}, we know that if $\DOM_t(G)=|V(G)|$ then $G$ must be a disjoint union of cycles hence $|E(G)|=|V(G)|$ for these graphs. One would expect that if the number of edges of $G$ is big enough then $\DOM_t(G)$ needs to be significantly smaller than the number of vertices.

Let us remark that if $G\in\cF_1\cup\cF_2\cup\cF_3$ then $|E(G)|\le 2|V(G)|-2$ must hold but equality can be reached in multiple ways. The easiest to see that if $G\in\cF_2$ and $s$ is adjacent to every other vertex then $|E(G)|=2|V(G)|-2$. We would like to finish by posing the following general problem.

\begin{problem}\label{p:edge}
    Can we establish some connection between $|E(G)|$ and $\DOM_t(G)$? For example, if $|E(G)|=O(n^2)$ then what could be an upper bound on $\DOM_t(G)$? Or in the other direction, can we find a graph $G$ such that $\DOM_t(G)$ is quite large although $|E(G)|$ is quite big?
\end{problem}

\section*{Acknowledgements}

This research was started at the beginning of 2024, and during that time the first author was supported by the ÚNKP-23-4-SZTE-628 New National Excellence Program, and the second author was supported by the ÚNKP-23-6-I-ELTE-1256 New National Excellence Program. Currently, the first author is supported by the EKÖP-24-4-SZTE-609 Program. All of these programs belong to the Ministry for Culture and Innovation from the source of the National Research, Development and Innovation Fund.

~\hfill \protect\includegraphics[height=0.8cm]{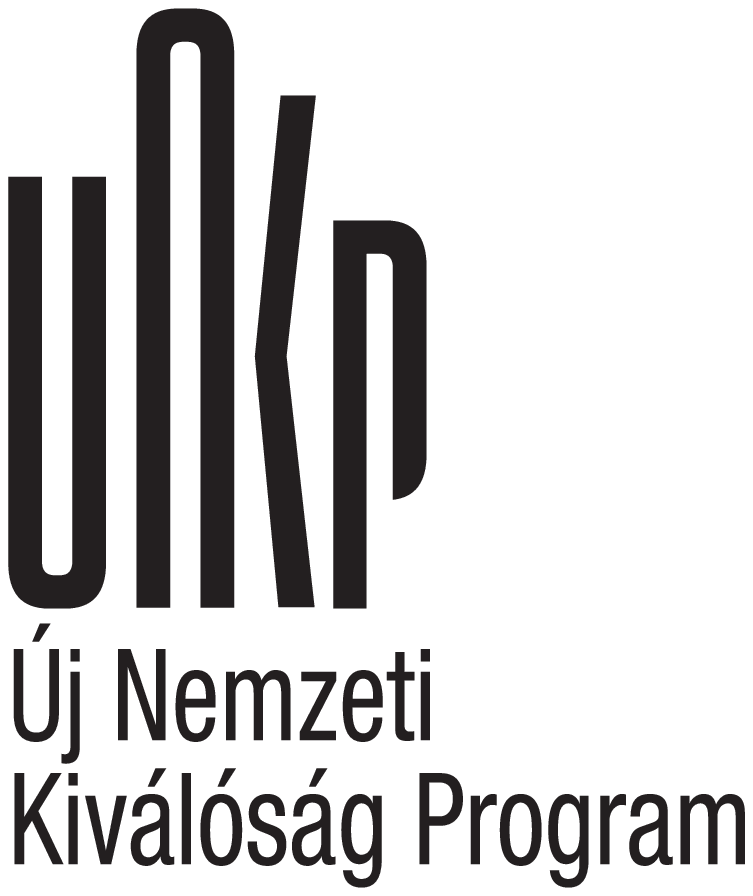}\includegraphics[height=0.8cm]{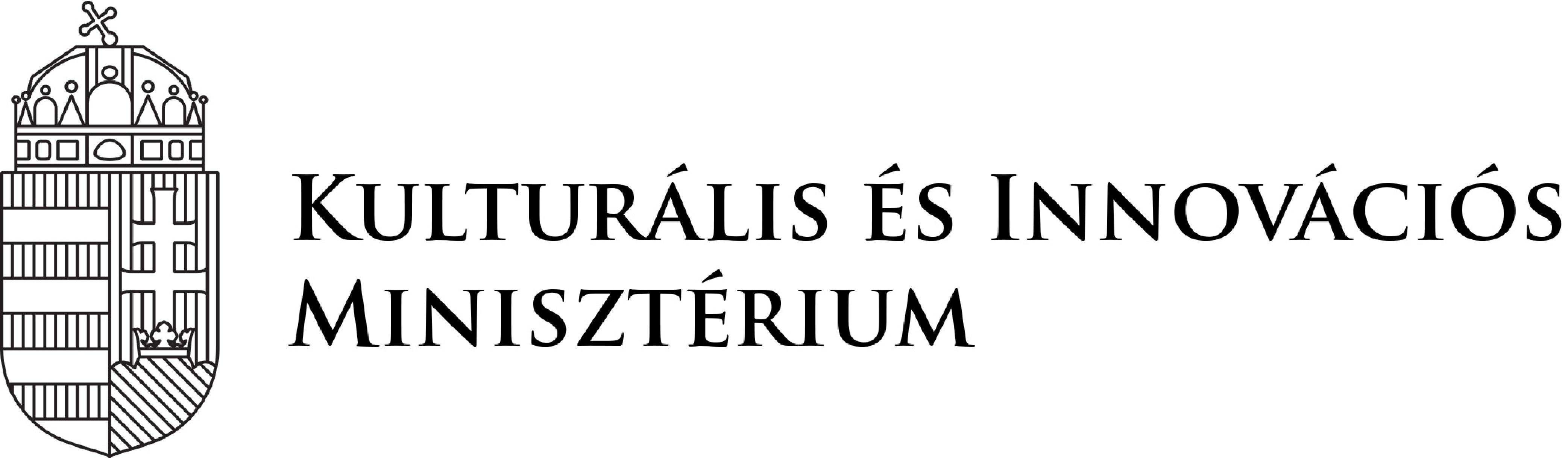} \includegraphics[height=0.8cm]{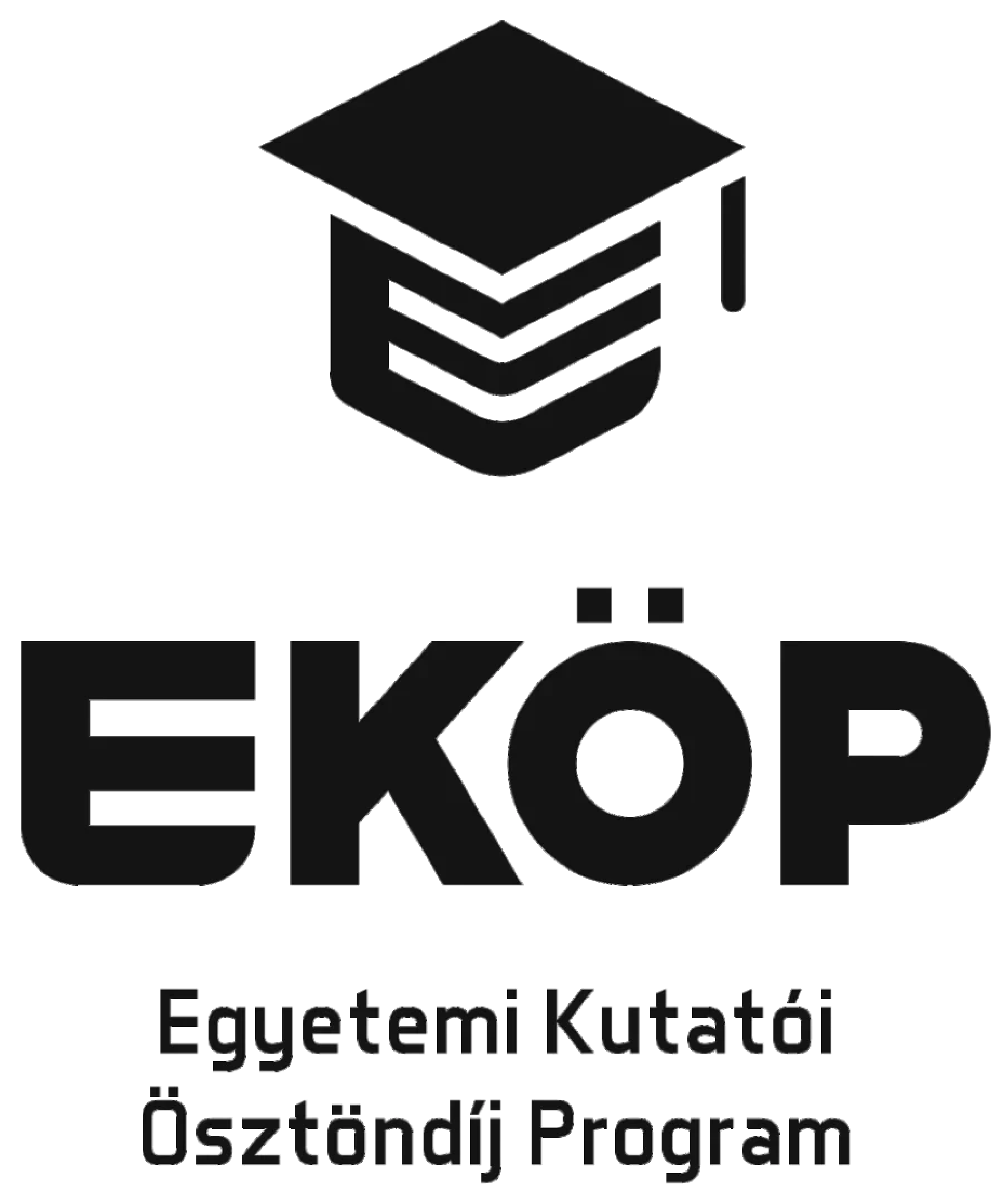}

\end{document}